\documentclass[letterpaper]{amsart}

\usepackage[letterpaper,left=1.5in,right=1.5in,top=1in,bottom=1in]{geometry}

\newcommand{\myauthor}{Benjamin Antieau and Ben Williams}
\newcommand{\mytitle}{Topology and purity for torsors}

\title{\mytitle}
\author{\myauthor}

\usepackage[pdfstartview=FitH,
            pdfauthor={\myauthor},
            pdftitle={\mytitle},
            colorlinks,
            linkcolor=reference,
            citecolor=citation,
            urlcolor=e-mail,
            ]{hyperref}

\usepackage{amsmath}
\usepackage{amscd}
\usepackage{amsbsy}
\usepackage{amssymb}
\usepackage{verbatim}
\usepackage{eufrak}
\usepackage{microtype}
\usepackage{hyperref}
\usepackage{mathrsfs}
\usepackage{amsthm}
\usepackage[all,cmtip]{xy}
\usepackage{tikz}
\usetikzlibrary{matrix,arrows}
\usepackage[abbrev,lite]{amsrefs}

\usepackage{mathpazo}

\usepackage{color}
\definecolor{todo}{rgb}{1,0,0}
\definecolor{conditional}{rgb}{0,1,0}
\definecolor{e-mail}{rgb}{0,.40,.80}
\definecolor{reference}{rgb}{.20,.60,.22}
\definecolor{mrnumber}{rgb}{.80,.40,0}
\definecolor{citation}{rgb}{0,.40,.80}

\DeclareMathAlphabet{\mathpzc}{OT1}{pzc}{m}{it}
\usepackage{dsfont}

\DeclareMathOperator{\PGL}{PGL}
\DeclareMathOperator{\BSL}{BSL}

\DeclareMathOperator{\SL}{SL}
\DeclareMathOperator{\GL}{GL}

\DeclareMathOperator{\tors}{tors}

\DeclareMathOperator{\Spec}{Spec}

\DeclareMathOperator{\Pic}{Pic}

\DeclareMathOperator{\Hoh}{H}
\DeclareMathOperator{\Eoh}{E}

\DeclareMathOperator{\Br}{Br}
\DeclareMathOperator{\W}{W}
\DeclareMathOperator{\I}{I}

\DeclareMathOperator{\im}{im}

\DeclareMathOperator{\ind}{ind}

\renewcommand{\P}{\operatorname{P}}

\newcommand{\topo}{{\mathrm{top}}}

\newcommand{\et}{\mathrm{\acute{e}t}}

\newcommand{\weq}{\simeq}

\newcommand{\B}{\mathrm{B}}

\newcommand{\colliot}{Colliot-Th\'{e}l\`{e}ne~}

\DeclareFontEncoding{OT2}{}{}

\newcommand{\we}{\simeq}
\newcommand{\iso}{\cong}

\newcommand{\CC}{\mathds{C}}

\newcommand{\QQ}{\mathds{Q}}
\newcommand{\ZZ}{\mathds{Z}}
\newcommand{\ZZp}{\mathds{Z}_{(p)}}

\newcommand{\PP}{\mathds{P}}

\newcommand{\Gm}{\mathds{G}_{m}}

\newcommand{\op}{\mathrm{op}}

\newcommand{\Ascr}{\mathscr{A}}

\newcommand{\Fscr}{\mathscr{F}}
\newcommand{\Oscr}{\mathscr{O}}

\let\oldmarginpar\marginpar
\renewcommand\marginpar[1]{\-\oldmarginpar[\raggedleft\footnotesize #1]%
{\raggedright\footnotesize #1}}

\newcommand{\tensor}{\otimes}

\newcommand{\Brm}{\mathrm{B}}

\newcommand{\BP}{\mathrm{BP}}
\newcommand{\BPGL}{\mathrm{BPGL}}

\newcommand{\Mrm}{\mathrm{M}}

\newcommand{\Lp}{\mathrm{L}_{(p)}}

\theoremstyle{plain}
\newtheorem{theorem}{Theorem}[section]
\newtheorem*{theorem*}{Theorem}
\newtheorem{lemma}[theorem]{Lemma}
\newtheorem{scholium}[theorem]{Scholium}
\newtheorem{proposition}[theorem]{Proposition}

\newtheorem{conjecture}[theorem]{Conjecture}
\newtheorem{corollary}[theorem]{Corollary}

\newtheoremstyle{named}{}{}{\itshape}{}{\bfseries}{.}{.5em}{#1 \thmnote{#3}}
\theoremstyle{named}

\theoremstyle{definition}

\newtheorem{example}[theorem]{Example}
\newtheorem{question}[theorem]{Question}

\theoremstyle{remark}

\begin{document}

\begin{abstract}
    We study the homotopy theory of the classifying space of the complex projective
    linear groups to prove that purity fails for $\PGL_p$-torsors on regular noetherian schemes when
    $p$ is a prime. Extending our previous work when $p=2$, we obtain a negative answer to a question of \colliot and
    Sansuc, for all $\PGL_p$. We also give a new example of the failure of purity for the
    cohomological filtration on the Witt group, which is the first example of this kind of a
    variety over an algebraically closed field.
\end{abstract}

\maketitle

\tableofcontents

\section{Introduction}

Let $X$ be a regular noetherian integral scheme, let $G$ be a smooth reductive group scheme over $X$, and let $K$ be the function field of
$X$. Consider the injective map
\begin{equation}\label{eq:purity}
    \im\left(\Hoh^1_{\et}(X,G)\rightarrow\Hoh^1_{\et}(\Spec K,G)\right)\rightarrow\bigcap_{x\in
    X^{(1)}}\im\left(\Hoh^1_{\et}(\Spec\Oscr_{X,x},G)\rightarrow\Hoh^1_{\et}(\Spec
    K,G)\right),
\end{equation}
where the intersection is over all codimension-$1$ points of $X$.
\colliot and Sansuc ask in~\cite{colliot-thelene-sansuc}*{Question 6.4} whether this map is
surjective. When it is, we say that \textit{purity\/} holds for $\Hoh^1_{\et}(X,G)$.

Purity trivially holds for $\Hoh^1_{\et}(X,G)$ when $G$ is special in the sense of Serre, for example if $G=\SL_n$, since
$\Hoh^1_{\et}(\Spec K,G)$ is a single point in this case. It holds for $\Hoh^1_{\et}(X,G)$
where $G$ is a finite type $X$-group scheme of multiplicative type
by~\cite{colliot-thelene-sansuc}*{Corollaire 6.9}.
It is also known to hold in many cases when $X$ is the spectrum of a regular local ring containing a
field of characteristic $0$. With this assumption, purity was proven for $\Hoh^1_{\et}(X,G)$ when $G$ is a
split group of type $\mathrm{A}_n$, a split orthogonal or special orthogonal group, or a split
spin group by
Panin~\cite{panin-purity} and also when $G=\mathrm{G}_2$
by Chernousov and Panin~\cite{chernousov-panin}. The local purity conjecture
asserts that purity holds for $\Hoh^1_{\et}(X,G)$ whenever $X$ is the spectrum of a regular noetherian
integral semi-local ring and $G$ is a smooth reductive algebraic $X$-group scheme.
Finally, purity holds for $\Hoh^1_{\et}(X,G)$ if the Krull dimension of $X$ is at most $2$
by~\cite{colliot-thelene-sansuc}*{Theorem 6.13}.

Purity is often considered along with another property, the so-called injectivity property,
which is said to hold when $\Hoh^1_{\et}(X,G)\rightarrow\Hoh^1_{\et}(U,G)$ has trivial
kernel for
all $U\subseteq X$ containing $X^{(1)}$. In fact,
Grothendieck and Serre conjectured that this map is always injective when $X$ is the
spectrum of a regular local ring $R$ and $G$ is a reductive $X$-group
scheme. This has been proved recently using affine Grassmannians by Fedorov and
Panin~\cite{fedorov-panin} when $R$ contains an infinite field following partial progress by
many other mathematicians. They prove more strongly that
$\Hoh^1_{\et}(X,G)\rightarrow\Hoh^1_{\et}(U,G)$ is injective.
The injectivity property for torsors is usually only sensible when $X$ is in fact the
spectrum of a local ring: otherwise it typically fails, even for $G=\Gm$. 

When $X$ is neither local nor low-dimensional and $G$ is a non-special semisimple algebraic
group, no results were known
about purity for torsors until our paper~\cite{aw4}, which showed that purity fails for
$\PGL_2$-torsors on smooth affine complex $6$-folds in general. It is the purpose of this
paper to use $p$-local homotopy theory to extend our previous result to $\PGL_p$ for all
$p$.

\begin{theorem}
    Let $p$ be a prime. Then, there exists a smooth affine complex variety $X$ of dimension
    $2p+2$ such that purity fails for $\Hoh^1_{\et}(X,\PGL_p)$.
\end{theorem}

We outline the proof. Recall first that $\PGL_p$-torsors correspond to degree-$p$ Azumaya
algebras, and write $\Br_{\topo}(X(\CC))$ for the
topological Brauer group, which classifies topological Azumaya algebras up to Brauer equivalence~\cite{grothendieck-brauer-1}. Let $X$ be a
smooth complex variety such that $\Hoh^2(X(\CC),\ZZ)=0$. In this case, by~\cite{aw3}*{Lemma 6.3}, there is an isomorphism
$\Br(X)\iso\Br_{\topo}(X(\CC))=\Hoh^3(X(\CC),\ZZ)_{\tors}$. Because $\Hoh^2(X(\CC),\ZZ)=0$,
topological Azumaya algebras of degree $n$ and exponent $m$
on $X(\CC)$ are classified by homotopy classes of maps $X(\CC)\rightarrow\BP(m,n)$, where $\P(m,n)=\SL_n(\CC)/\mu_m$.

In order to prove the theorem, we must construct a complex affine variety, $X$. First, following Totaro~\cite{totaro},
we take a high dimensional algebraic approximation, $X$, to the classifying space $\BP(p,pq)$, where $q>1$ is prime to
$p$. This $X$ is equipped with an $\SL_{pq}/\mu_p$-torsor, which induces a $\PGL_{pq}$-torsor and therefore an Azumaya
algebra $A$. Let $\alpha$ be the Brauer class of $A$ on $X$. The exponent of $\alpha$ is $p$. Comparing the $p$-local
homotopy type of $\BP(p,pq)$ to that of $\BPGL_p(\CC)$, we find that there is a non-vanishing topological obstruction in
$\Hoh^{2p+2}(X(\CC),\ZZ/p)$ to the existence of a degree-$p$ Azumaya algebra on $X$ with the same Brauer class as
$A$. Second, we replace $X$ by a homotopy-equivalent smooth affine variety using Jouanolou's
device~\cite{jouanolou}. Third, we use the affine Lefschetz theorem~\cite{goresky-macpherson}*{Introduction, Section
  2.2} to cut down to a smooth affine $2p+2$-dimensional variety where the obstruction in $\Hoh^{2p+2}(X(\CC),\ZZ/p)$
persists. By using an unpublished preprint of Ekedahl~\cite{ekedahl}, it is possible to construct smooth projective
complex examples of this nature as well, although we will not emphasize this last point in our paper.

Let $K$ be the function field of the $2p+2$--dimensional affine variety $X$ alluded to in
the previous paragraph. The theorem is deduced from
the properties of $X$ as follows. The Brauer class $\alpha_K\in\Br(K)$ has exponent $p$ and
index dividing $pq$. Its index is therefore $p$ by a result of
Brauer~\cite{gille-szamuely}*{Proposition 4.5.13},
and it is represented by a division algebra $D$ of degree $p$ over $K$. If $P\in X^{(1)}$,
then $\alpha$ restricts to a class $\alpha_P\in\Br(\Oscr_{X,p})$. Since $D$ is unramified
along $\Oscr_{X,P}$ and since $\Oscr_{X,P}$ is a discrete valuation ring, it follows that any
maximal order in $D$ over $\Oscr_{X,P}$ is in fact an Azumaya algebra (see the proof
of~\cite{auslander-goldman}*{Proposition 7.4}).
Thus, the class of $D$ is in the target of the map of~\eqref{eq:purity}, but by our choice of $X$,
the class of $D$ is not in the source of that map.

We make the following conjecture.
\begin{conjecture}\label{conj:failure}
    Let $G$ be a non-special semisimple $k$-group scheme. Then there exists a smooth
    affine $k$-variety $X$ such that purity fails for $\Hoh^1_{\et}(X,G)$.
\end{conjecture}

Our theorem proves the conjecture for $G=\PGL_p$ over $\CC$, and since the schemes in question may all be defined over $\QQ$, the conjecture
is actually settled for $\PGL_p$ over any field of characteristic $0$.

We explore three other points in the paper. First, in Section~\ref{sec:ojanguren} we show
that, in contrast to the global case, purity holds for $\PGL_n$-torsors over regular
noetherian integral semi-local rings $R$, at least if we restrict our attention to those torsors
whose Brauer class has exponent invertible in $R$. This is a generalization of equivalent
results of Ojanguren~\cite{ojanguren} and Panin~\cite{panin-purity} in characteristic $0$.

Second, in Section~\ref{sec:canonical}, we give a topological perspective that explains why we expect purity to fail for
$\Hoh^1_{\et}(X,\PGL_n)$ for all $n$.

Third, in Section~\ref{sec:witt}, we give examples where purity fails for $\I^2(X)/\I^3(X)$ where
$\I^\bullet$ is the filtration on the Witt group induced by the cohomological filtration on
$\W(\CC(X))$ and $X$ is a certain smooth affine complex
$5$-fold. Previous examples of a different, arithmetic nature were produced by Parimala and Sridharan~\cite{parimala-sridharan-2}, but these
were explained by Auel~\cite{auel} as failing to take into account quadratic modules with coefficients in line bundles. Our examples have
$\Pic(X)=0$.

The authors would like to thank Asher Auel, Eric Brussel, Roman Fedorov, Henri Gillet, Max
Lieblich, and Manuel
Ojanguren for conversations related to the present work. We would also like to thank Burt
Totaro for several useful comments on an early version of this paper.

\section{Topology}\label{sec:topology}

In~\cite{aw4}, we used knowledge of both the low-degree singular cohomology of $\BPGL_2(\CC)$ and of the low-degree Postnikov tower of
$\BPGL_2(\CC)$ to produce counterexamples to the existence of Azumaya maximal orders in unramified division algebras.
This is equivalent to showing that purity fails for $\PGL_2$ over $\CC$. At the time we wrote~\cite{aw4}, we did not know
how to extend our results to other primes, because our argument relied on the accessibility of the low-degree Postnikov tower of
$\BPGL_2(\CC)$. While remarkable calculations have been made by Vezzosi~\cite{vezzosi} and Vistoli~\cite{vistoli} on the cohomology of
$\BPGL_p(\CC)$ for odd primes $p$, the problem of determining the Postnikov tower up the necessary level, $2p+1$, was beyond us. By using a
$p$-local version of our arguments in~\cite{aw4} we bypass our ignorance to prove similar results.

We prove a result in this section about self-maps of $\tau_{\leq 2p+1}\BPGL_p(\CC)$, the $2p+1$ stage in the Postnikov tower of
$\BPGL_p(\CC)$.  Our theorem is in some sense related to the important results of Jackowski, McClure, and
Oliver~\cite{jackowski-mcclure-oliver-1} about maps $\mathrm{BG}\rightarrow\mathrm{BH}$ when $\mathrm{G}$ and $\mathrm{H}$ are compact Lie
groups, and especially about self-maps of $\mathrm{BG}$. For the applications to algebraic geometry we have in mind, one must use finite
approximations to $\BPGL_p(\CC)\we\mathrm{BPU}_p$, where the results of~\cite{jackowski-mcclure-oliver-1} do not immediately apply. For more on the
relationship of our work to~\cite{jackowski-mcclure-oliver-1}, see Section~\ref{sec:canonical}.

The group $\PGL_p(\CC)$ and other classical groups are always equipped with the classical
topology.

\subsection{The \texorpdfstring{$p$}{p}-local cohomology of some Eilenberg-MacLane spaces}

We fix a prime number $p$. The $p$-local cohomology of a space $X$ is the singular cohomology of $X$ with coefficients in $\ZZp$. In the
next few lemmas, we compute the low-degree $p$-local cohomology of $K(\ZZ,n)$, up to the first $p$-torsion. These results are both
straightforward and classical, being corollaries of the all-encompassing calculations of
Cartan~\cite{cartan} for instance. We include proofs here for the sake of completeness.

Recall that $K(\ZZ,2)\we\CC\PP^\infty$ and that $\Hoh^*(K(\ZZ,2),\ZZ)\iso \ZZ[\iota_2]$, where
$\deg(\iota_2)=2$. In general, there is a canonical class $\iota_n\in\Hoh^n(K(\ZZ,n),\ZZ)$
representing the identity map. We will use the Serre spectral sequences for the fiber sequences
$K(\ZZ,n)\rightarrow\ast\rightarrow K(\ZZ,n+1)$ as well as the multiplicative structure in the spectral
sequences.

\begin{lemma}
    For $1\leq k\leq 2p+4$, the $p$-local cohomology of $K(\ZZ,3)$ is
    \begin{equation*}
        \Hoh^k(K(\ZZ,3),\ZZp)\iso\begin{cases}
            \ZZp   &   \text{if $k=3$,}\\
            \ZZ/p       &   \text{if $k=2p+2$,}\\
            0           &   \text{otherwise.}
        \end{cases}
    \end{equation*}
    \begin{proof}
        We can choose $\iota_3$ so that $d_3(\iota_2)=\iota_3$ in the Serre spectral
        sequence for $K(\ZZ,2)\rightarrow\ast\rightarrow K(\ZZ,3)$:
        \begin{equation*}
            \Eoh_2^{s,t}=\Hoh^s(K(\ZZ,3),\Hoh^t(K(\ZZ,2),\ZZp))\Rightarrow\Hoh^{s+t}(\ast,\ZZ_{(p)}).
        \end{equation*}
        Then,
        $d_3(\iota_2^n)=n\iota_2^{n-1}\iota_3$ and it follows that $d_3(\iota_2^n)$ is a
        generator of $\Eoh_3^{3,2n-2}$ for $1\leq n<p$. For $4\leq k\leq
        2p+1$ the cohomology group $\Hoh^k(K(\ZZ,3),\ZZp)$ vanishes since there are no
        possible non-zero differentials hitting it. The first point on the $t$-axis
        where $d_3$ is not surjective is $d_3:\Eoh_3^{0,2p}\rightarrow\Eoh_3^{3,2p-2}$ where the cokernel is $\ZZ/p$, see
        Figure~\ref{fig:ssskz3}.
        \begin{figure}[h]
            \centering
            \begin{equation*}
                \xymatrix@R=3pt{
                    \ZZp\cdot\iota_2^{p+1}\ar[ddrrr]^{d_3}  &   0   &   0   &
                    \ZZp\cdot\iota_2^{p+1}\iota_3\\
                    0&0&0&0\\
                    \ZZp\cdot\iota_2^p\ar[ddrrr]^{d_3}&0&0&\ZZp\cdot\iota_2^{p}\iota_3\\
                    0&0&0&0\\
                    \ZZp\cdot\iota_2^{p-1}&0&0&\ZZp\cdot\iota_2^{p-1}\iota_3\\
                    0&0&0&0\\
                    \vdots&\vdots&\vdots&\vdots\\
                    0&0&0&0\\
                    \ZZp\cdot\iota_2\ar[ddrrr]^{d_3}&0&0&\ZZp\cdot\iota_2\iota_3\\
                    0&0&0&0\\
                    \ZZp&0&0&\ZZp\cdot\iota_3\\
                }
            \end{equation*}
            \caption{The $\Eoh_3$-page of the Serre spectral sequence associated to $K(\ZZ,2)\rightarrow\ast\rightarrow K(\ZZ,3)$.}
            \label{fig:ssskz3}
        \end{figure}
        In order for the sequence to converge to zero, this non-zero cokernel must support a
        non-zero departing differential; since $\Hoh^k(K(\ZZ,2),\ZZp)=0$ for $4\leq k\leq 2p+1$, the
        differential $d_{2p-1}$ induces an isomorphism $\ZZ/p\rightarrow\Hoh^{2p+2}(K(\ZZ,3),\ZZp)$.
        Let $j_p$ be a generator of $\Hoh^{2p+2}(K(\ZZ,3),\ZZp)$.
        In terms of total degree, the next non-zero term in the spectral sequence is
        $\Eoh_3^{2,2p+2}=\ZZ/p\cdot\iota_2j_p$. Thus, the next potentially non-zero
        $p$-local cohomology group of $K(\ZZ,3)$ is $\Hoh^{2p+5}(K(\ZZ,3),\ZZp)$.
    \end{proof}
\end{lemma}

The next two lemmas have proofs conceptually similar to the preceding proof.

\begin{lemma}\label{lem:kz4}
    For $0\leq k\leq 2p+5$, the $p$-local cohomology of $K(\ZZ,4)$ is
    \begin{equation*}
        \Hoh^k(K(\ZZ,4),\ZZp)\iso\begin{cases}
            \ZZp    &   \text{if $k=0\,\mathrm{mod}\,4$,}\\
            \ZZ/p   &   \text{if $k=2p+3$,}\\
            0       &   \text{otherwise.}
        \end{cases}
    \end{equation*}
    \begin{proof}
        Again, we may assume that $d_3(\iota_3)=\iota_4$ in the Serre spectral sequence.
        Then, $d_3(\iota_3\iota_4^n)=\iota_4^{n+1}$. Moreover, the powers of $\iota_4$ are
        non-zero because the $d_3$ differential leaving $\Eoh_3^{4n,3}=\ZZp\cdot\iota_3\iota_4^n$ cannot have a kernel as all
        cohomology of $K(\ZZ,3)$ in degrees higher than $3$ is torsion. For this reason the group
        $\Hoh^{2p+2}(K(\ZZ,3),\ZZp)$ survives to the
        $\Eoh_{2p+3}$-page of the spectral sequence, and the differential
        $$d_{2p+3}:\ZZ/p\iso\Hoh^{2p+2}(K(\ZZ,3),\ZZp)\rightarrow\Hoh^{2p+3}(K(\ZZ,4),\ZZp)$$
        is an isomorphism. The next potential non-zero torsion class in the spectral
        sequence is in $\Hoh^{2p+5}(K(\ZZ,3),\ZZp)$, which shows that the other cohomology
        groups vanish in the range indicated.
    \end{proof}
\end{lemma}

\begin{lemma}\label{lem:kzn}
    The cohomology groups $\Hoh^k(K(\ZZ,n),\ZZp)$ are torsion-free for $0\leq k\leq 2p+3$
    and $n\geq 5$. In this range they are isomorphic to a polynomial algebra over
    $\ZZp$ with a single generator $\iota_n$ in degree $n$ if $n$ is even or an exterior algebra over
    $\ZZp$ with a single generator $\iota_n$ in degree $n$ if $n$ is odd.
    \begin{proof}
        This follows inductively as in the previous two lemmas. The important point is that the first
        $p$-torsion in $\Hoh^*(K(\ZZ,n-1),\ZZp)$ is in degree $2p+(n-1)-1$. No differential
        exiting it can be non-zero until the differential $d_{2p+(n-1)}$, which produces
        $p$-torsion in $\Hoh^{2p+n-1}(K(\ZZ,n),\ZZp)$. If $n\geq 5$, then $2p+n-1\geq 2p+4$.
    \end{proof}
\end{lemma}

\subsection{The \texorpdfstring{$p$}{p}-local homotopy type of \texorpdfstring{$\BSL_p(\CC)$}{BSLpC}}

Now we harness the computations of the previous section to study the $p$-local homotopy type
of truncations of $\BPGL_p(\CC)$. If $X$ is a connected topological space, we will write
$\tau_{\leq n}X$ for the $n$th stage in the Postnikov tower of a path-connected space $X$. Thus, $\tau_{\leq n}X$
is a topological space such that
\begin{equation*}
    \pi_i\left(\tau_{\leq n}X\right)\iso\begin{cases}
        \pi_i(X) &   \text{if $i\leq n$,}\\
        0       &   \text{otherwise.}
    \end{cases}
\end{equation*}
The Postnikov tower is the sequence of natural maps
\begin{equation*}
    \xymatrix@R=10pt{
         X \ar[rd] \ar[rdd] \ar@/_5pt/[rdddd] \ar@/_20pt/[rddddd]  &   \vdots\ar[d]\\
           &   \tau_{\leq n}X\ar[d]\\
           &   \tau_{\leq n-1}X\ar[d]\\
           &   \vdots\ar[d]\\
           &   \tau_{\leq 1}X\ar[d]\\
           &   \tau_{\leq 0}X = \ast
    }
\end{equation*}
with the fiber of $\tau_{\leq n}X\rightarrow\tau_{\leq n-1}X$ identified with $K(\pi_nX,n)$.  In good cases, such as when the action of
$\pi_1X$ on $\pi_nX$ for $n\ge 1$ is trivial, the extension
\begin{equation*}
   \xymatrix{ K(\pi_nX,n)\ar[r] & \tau_{\leq n}X \ar[d] \\ & \tau_{\leq n-1}X }
\end{equation*}
is classified by the $k$-invariant 
\[k_{n-1}:\tau_{\leq n-1}X\rightarrow K(\pi_nX,n+1)\]
 in the sense that $\tau_{\leq n}X$ is the homotopy fiber of $k_{n-1}$. This $k$-invariant is a
cohomology class in $\Hoh^{n+1}(\tau_{\leq n-1}X,\pi_nX)$. When it vanishes, the fibration is trivial.

There is a $p$-localization functor $\Lp$ that takes a topological space $X$ and produces a space $\Lp X$ whose homotopy groups are
$\ZZp$-modules. For the theory of localization of CW complexes, we refer to the monograph of Bousfield and Kan~\cite{bousfield-kan}. This
functor takes fiber sequences to fiber sequences when the base is simply connected by the principal fibration
lemma~\cite{bousfield-kan}*{Chapter II}. Since the $\ZZp$-localization of an Eilenberg-MacLane space $K(\pi,n)$ is
$K(\pi\otimes_\ZZ\ZZp,n)$, for which see~\cite{bousfield-kan}*{page 65},
it follows that application of $\Lp$ commutes with the formation of Postnikov towers of simply-connected spaces.  

Now, we consider the $p$-local homotopy type of certain stages in the Postnikov tower of
$\BSL_n(\CC)$. By Bott periodicity~\cite{bott}*{Theorem 5} the $p$-local homotopy groups of $\BSL_n(\CC)$
for $1\leq i\leq 2n+1$ are
\begin{equation*}
    \pi_i\left(\Lp\BSL_n(\CC)\right)\iso\pi_i\left(\BSL_n(\CC)\right)\otimes_\ZZ\ZZp\iso\begin{cases}
        \ZZp    &   \text{if $i$ is even and $i\geq 4$,}\\
        \ZZ/(n!)\otimes_\ZZ\ZZp &   \text{if $i=2n+1$,}\\
        0   &   \text{otherwise.}
    \end{cases}
\end{equation*}

\begin{proposition} \label{p:SLsplit}
    The localization $\Lp\tau_{\leq 2p}\BSL_n(\CC)$, where $n\geq p$, is a generalized Eilenberg--MacLane space:
    \begin{equation} \label{eq:GEMS}
        \Lp\tau_{\leq 2p}\BSL_n(\CC)\we K(\ZZp,4)\times K(\ZZp,6)\times\cdots\times K(\ZZp,2p).
    \end{equation}
    \begin{proof}
        We prove the general statement 
  \begin{equation*} 
        \Lp\tau_{\leq 2j}\BSL_n(\CC)\we K(\ZZp,4)\times K(\ZZp,6)\times\cdots\times K(\ZZp,2j), \quad \text{ for $j \le p$}
    \end{equation*}
 by induction on $j$. The base case when $j=1$ is trivial.
        For the induction step, suppose that $\Lp\tau_{\leq
        2j}\BSL_n(\CC)$ is
        \begin{equation*}
            K(\ZZp,4)\times\cdots\times K(\ZZp,2j)
        \end{equation*}
        for some $1\leq j< p$. The extension
        \begin{equation*}
            K(\ZZp,2j+2)\rightarrow\Lp\tau_{\leq 2j+2}\BSL_n(\CC)\rightarrow\Lp\tau_{\leq 2j}\BSL_n(\CC)
        \end{equation*}
        is classified by the $k$-invariant
        \begin{equation*}
            k_{2j}\in\Hoh^{2j+3}(K(\ZZp,4)\times\cdots\times K(\ZZp,2j),\ZZp).
        \end{equation*}
        By Lemmas~\ref{lem:kz4} and~\ref{lem:kzn}, this cohomology group must vanish, since
        $j<p$. Hence $k_{2j}=0$ and the extension is trivial.
    \end{proof}
\end{proposition}

Before we prove the next proposition, we need a well-known lemma.
Recall that an $n$-equivalence is a map such that $\pi_k(f):\pi_k(X)\rightarrow\pi_k(Y)$
is an isomorphism for $0\leq k <n$ and a surjection for $k=n$.

\begin{lemma} \label{lem:nequivcoho}
    Let $f:X\rightarrow Y$ be an $n$-equivalence. Then, for any coefficient abelian group $A$, the
    induced map
    \begin{equation*}
        f^*:\Hoh^k(Y,A)\rightarrow\Hoh^k(X,A)
    \end{equation*}
    is an isomorphism for $0\leq k\leq n-1$ and an injection for $k=n$.
    \begin{proof}
         This follows most easily from the Serre spectral sequence for the fibration sequence
        $F\rightarrow X\rightarrow Y$. Since the fiber is $n$-connected, the groups $\widetilde \Hoh^k(F,A)$ vanish for $k<n$. The first
        nontrivial extension-problem in the spectral sequence takes the form
        \[ 0 \to \Hoh^n(Y, A) \to \Hoh^n(X, A) \to \Hoh^n(F, A) \]
        which proves the result.
    \end{proof}
\end{lemma}

The previous proposition asserts that $\Lp\tau_{\le 2p} \BSL_p(\CC)$ is a generalized Eilenberg--MacLane space, the following asserts that
the $\tau_{\le 2p}$ appearing there is sharp, and the nontriviality of the extension can be detected after pulling the extension back along
an inclusion $K(\ZZp, 4) \to \Lp\tau_{\le 2p} \BSL_p(\CC)$.

\begin{proposition}\label{prop:nonsplit}
    Denote by $i$ a map $i: K(\ZZp,4)\rightarrow\Lp\tau_{\leq 2p}\BSL_p(\CC)$ splitting the
    projection map.
    Write $k_{2p}\in\Hoh^{2p+2}(\Lp\tau_{\leq 2p}\BSL_p(\CC),\ZZ/p)$ for the $k$--invariant of the extension
    \begin{equation} \label{eq:nontrivext}
       \xymatrix{ K(\ZZ/p,2p+1)\ar[r] & \Lp\tau_{\leq 2p+1}\BSL_p(\CC) \ar[d] \\ & \Lp\tau_{\leq
        2p}\BSL_p(\CC).}
    \end{equation}
    Then $k_{2p}$ is of order $p$, and moreover $i^*(k_{2p})$ is a generator for $\Hoh^{2p+2}( K(\ZZp, 4) , \ZZ/p) \iso \ZZ/p$.    
\end{proposition}
\begin{proof}
      Note that $X \to \tau_{\le n} X$ is an $(n+1)$--equivalence. By Lemma \ref{lem:nequivcoho}, the map of rings
      \[ \Hoh^{i}( \Lp \tau_{\le 2p} \BSL_p, \ZZp) \to  \Hoh^{i} ( \Lp \BSL_p, \ZZp) = \ZZp[ c_2, c_3 , \dots, c_{p}]  \]
      is an isomorphism when $i \le 2p$, and an injection, and hence an isomorphism, when $i = 2p+1$.
      By Lemmas \ref{lem:kz4} and \ref{lem:kzn}, the ring $\Hoh^i(\Lp \tau_{\le 2p} \BSL_p, \ZZp)$ is isomorphic to a polynomial ring on generators
      in degrees $4, 6, 8, \dots, 2p$ in the range where $i \le 2p+2$, so that it follows that
      \[ \Hoh^{2p+2}( \Lp \tau_{\le 2p} \BSL_p, \ZZp) \to \Hoh^{2p+2}( \Lp \BSL_p, \ZZp) \] is an isomorphism as well. We also deduce that
      \[\Hoh^{2p+3}(\Lp \tau_{\le 2p} \BSL_p(\CC), \ZZp) \iso \ZZ/p \cdot \rho\]
      where $i^*(\rho)$ is a generator of $\Hoh^{2p+3}( K(\ZZ, 4), \ZZp) \iso \ZZ/p.$

      Considering the long exact sequence in cohomology associated to the sequence
      \[ 0 \to \ZZp \to \ZZp \to \ZZ/p \to 0 \]
      we deduce the existence of a decomposition
      \[ \Hoh^{2p+2}( \Lp \tau_{\le 2p} \BSL_p, \ZZ/p) = \Hoh^{2p+2}(\Lp \BSL_p, \ZZ/p) \oplus \ZZ/p \cdot \sigma \]
      where $\beta_p(\sigma) = \rho$.
      
      We observe two things. First that there is a quotient relationship arising from the Postnikov extensions
      \[ \Hoh^{2p+2}( \Lp \tau_{\le 2p} \BSL_p, \ZZ/p) / \langle k_{2p} \rangle \iso
          \Hoh^{2p+2}(\Lp \tau_{\le 2p+1} \BSL_p, \ZZ/p), \]
      and second that the functorial map
      \[ \Hoh^{2p+2}( \Lp \tau_{\le 2p+1} \BSL_p, \ZZ/p) \hookrightarrow \Hoh^{2p+2}( \Lp \BSL_p, \ZZ/p) \]
      is injective by Lemma \ref{lem:nequivcoho}. It follows directly that $k_{2p} = u \sigma$ where $u$ is a unit. By naturality, $\beta_{2p}( i^*(k_{2p})) = u
      i^*(\rho)$, and in particular, $i^*(k_{2p}) \neq 0$.
\end{proof}

\begin{corollary} \label{l:lemma2}
    A map $h: \Lp \tau_{\leq 2p+1} \BSL_p(\CC) \to  \Lp \tau_{\leq 2p+1} \BSL_p(\CC)$ that
    induces an isomorphism on $\pi_4\left(\Lp \tau_{\leq 2p+1} \BSL_p(\CC)\right) \iso
    \ZZ_{(p)}$, also induces an isomorphism on $$\pi_{2p+1}\left(\Lp \tau_{\leq 2p+1}
    \BSL_p(\CC)\right) \iso \ZZ/p.$$
\end{corollary}
  \begin{proof}
        Let $i: K(\ZZ_{(p)}, 4) \to \Lp \tau_{\leq 2p} \BSL_p(\CC)$ again denote a map splitting the
        projection onto $K(\ZZ_{(p)}, 4)$ in Proposition \ref{p:SLsplit}.  Let
        \begin{equation*}
            \tau_{\leq 2p}h:\tau_{\leq 2p}\BSL_p(\CC)\rightarrow\tau_{\leq 2p}\BSL_p(\CC)
        \end{equation*}
        be the truncation of $h$. This map fits into a commutative diagram
        \begin{equation} \label{e:cohOp}
        \xymatrix{
            K(\ZZ_{(p)}, 4) \ar^>>>>>i[r] \ar^{\simeq}[d] & \Lp \tau_{\leq 2p} \BSL_p( \CC) \ar^{k_{2p}}[r] \ar^{\tau_{\leq 2p}h}[d]  & K(\ZZ/p, 2p+2)  \ar^{\B h_*}[d]\\ 
            K(\ZZ_{(p)}, 4) \ar^>>>>>i[r] & \Lp \tau_{\leq 2p} \BSL_p(\CC) \ar^{k_{2p}}[r] & K(\ZZ/p, 2p+2),}
        \end{equation}
        where the map $\B h_*$ is the result of applying a functorial
        classifying-space construction to the endomorphism of $K(\pi_{2p+1}\left( \Lp
        \BSL_p\right) , 2p+1)
        \we K( \ZZ/p,  2p+1)$ arising from the map $h_*$ on $\pi_{2p+1} \Lp
        \BSL_p$. The map $K(\ZZ_{(p)}, 4) \to K(\ZZ_{(p)}, 4)$ is the
        composition of $i$ with $h$ and the projection, and is a weak equivalence
        since $i$, $h$ and the projection all induce isomorphisms on $\pi_4$, by
        hypothesis. Since $i^*(k_{2p})\neq 0$ is a generator of
        $\Hoh^{2p+2}(K(\ZZ_{(p)},4),\ZZ/p))$, commutativity of the diagram proves that $h_*$ is an equivalence, as claimed.
    \end{proof}

\subsection{The \texorpdfstring{$p$}{p}-local homotopy type of \texorpdfstring{$\BPGL_p(\CC)$}{BPGLp(C)}}

There is a fiber sequence, obtained by truncating a sequence associated to the
defining quotient $\SL_p(\CC)/ \mu_p = \PGL_p(\CC)$, of the form
\[ \xymatrix{ \tau_{\le 2p+1} \BSL_p(\CC) \ar[r] & \tau_{\le 2p+1} \BPGL_p(\CC) \ar[r] &
K(\ZZ/p, 2). }\]
The main theorem concerns itself with maps $f: \tau_{\le 2p +1} \BPGL_p(\CC) \to
\tau_{\le 2p+1} \BPGL_p(\CC)$ that induce isomorphisms on
$\pi_2(\tau_{2p+1}\BPGL_p(\CC))$, these maps fit into diagrams
\[ \xymatrix{ \tau_{\le 2p+1} \BSL_p(\CC) \ar@{-->}^{\tilde f}[d] \ar[r] & \tau_{\le 2p+1} \BPGL_p(\CC)
  \ar[r] \ar[d] & K(\ZZ/p, 2) \ar@{-}^{\weq}[d] \\ 
 \tau_{\le 2p+1} \BSL_p(\CC) \ar[r] & \tau_{\le 2p+1} \BPGL_p(\CC) \ar[r] & K(\ZZ/p, 2) }\]
in which the right-hand square commutes up to homotopy. The map, $\tilde f$,
making the left-hand square commute up to homotopy exists, but is not unique. We
refer to such a map as a lift of the map $f$.

Since $\pi_{2p+1}\left(\BSL_p(\CC)\right)\iso\ZZ/(p!)$, it follows that
$\pi_{2p+1}\left(\BPGL_p(\CC)\right)\iso\ZZ/(p!)$, and hence that
\[\pi_{2p+1}\left(\Lp\BPGL_p(\CC)\right) \iso \ZZ/(p!) \tensor_\ZZ \ZZ_{(p)} \iso \ZZ/p. \]

The following lemma is a technical ingredient in Theorem \ref{thm:p}.

\begin{lemma} \label{l:lemma1}
    Let $f: \tau_{\le 2p+1} \BPGL_p(\CC) \to \tau_{\le 2p+1} \BPGL_p(\CC)$ be a map that induces an isomorphism
    \[\pi_2(f):\pi_2\left(\tau_{\leq
    2p+1}\BPGL_p(\CC)\right)\rightarrow\pi_2\left(\tau_{\leq 2p+1}\BPGL_p(\CC)\right)=\ZZ/p.\]
    Any lift $\tilde f: \tau_{\le 2p+1} \BSL_p(\CC) \to \tau_{\le 2p+1} \BSL_p(\CC)$ of $f$ has the property that the
    $p$-localization
    \[ \pi_4(\Lp \tilde f_*):  \pi_4\left(\BSL_p(\CC)\right) \tensor_\ZZ \ZZ_{(p)} \to
        \pi_4\left(\BSL_p(\CC)\right) \tensor_\ZZ \ZZ_{(p)}\]
    is an isomorphism.
\end{lemma}
 \begin{proof}
        The first nontrivial fiber sequence appearing in the Postnikov tower of
        $\BPGL_p(\CC)$ is
        \[ \xymatrix{ K(\ZZ, 4) \ar[r] & \tau_{\le 4}\BPGL_p(\CC) \ar[r] & K(\ZZ/p,
        2). } \]
        In~\cite{aw3} we proved that the $K(\ZZ,4)$--bundle above is classified by a
        map $K(\ZZ/p, 2) \to K(\ZZ, 5)$ that represents a generator of the group
        $\Hoh^5(K(\ZZ/p, 2), \ZZ) \iso \ZZ/p^\epsilon$, where $\epsilon =1$ unless
        $p=2$ in which case $\epsilon=2$. If $f$ induces an
        isomorphism on $\pi_2$, it must also induce a map
        $f_*$ on $\pi_4\left(\tau_{\leq 2p+1}\BPGL_p(\CC)\right) \iso \ZZ$ such that the functorially-derived diagram
        \[ \xymatrix{ K(\ZZ/p, 2) \ar[d]^{\iso} \ar[r] & K(\ZZ, 5) \ar^{\B f_*}[d] \\
        K(\ZZ/p, 2) \ar[r] & K(\ZZ, 5) } \]
        commutes. The class $\B f_* \in \Hoh^5(K(\ZZ, 5), \ZZ) = \ZZ$ must 
        be an integer that is relatively prime to $p$, and in turn the endomorphism
        induced by $f$ on $\pi_4\left( \tau_{\leq 2p+1} \BPGL_p(\CC)\right) \iso \ZZ$ must be
        multiplication by an integer that is relatively prime to $p$. The map induced by
        $f$ on $\pi_4\left( \Lp \tau_{\leq 2p+1} \BPGL_p(\CC)\right)$ is consequently an isomorphism.

        For any choice of $\tilde f$, the $p$-localized diagram
        \[ \xymatrix{ \Lp \tau_{\leq 2p+1} \BSL_p(\CC) \ar^{\Lp \tilde f}[d] \ar[r] & \Lp
        \tau_{\leq 2p+1}
        \BPGL_p(\CC) \ar^{\Lp f}[d] \\ \Lp \tau_{\leq 2p+1} \BSL_p(\CC)  \ar[r] & \Lp
        \tau_{\leq 2p+1}
        \BPGL_p(\CC) } \]
        commutes. Here the horizontal arrows induce isomorphisms on all homotopy
        groups, $\pi_i$, where $i\ge 3$, and the result follows.
    \end{proof}

We are now in a position to prove the main topological theorem of the paper.

\begin{theorem}\label{thm:p}
    Let $f:\tau_{\leq 2p+1}\BPGL_p(\CC)\rightarrow\tau_{\leq 2p+1}\BPGL_p(\CC)$ be a 
    map that induces an isomorphism 
        \[\pi_2(f):\pi_2\left(\tau_{\leq  2p+1}\BPGL_p(\CC)\right)\to\pi_2\left(\tau_{\leq
        2p+1}\BPGL_p(\CC)\right)\iso\ZZ/p.\] 
        Then \[\pi_{2p+1}(\Lp f):\pi_{2p+1}\left(\Lp\tau_{\leq
          2p+1}\BPGL_p(\CC)\right) \to \pi_{2p+1}\left(\Lp\tau_{\leq 2p+1}\BPGL_p(\CC)\right)\] is an isomorphism.
    \begin{proof}
        Suppose $f$ is a map meeting the hypothesis of the theorem. Choose a lift,
        $\tilde f : \tau_{\leq 2p+1} \BSL_p(\CC) \to \tau_{\leq 2p+1} \BSL_p(\CC)$. By
        Lemma~\ref{l:lemma1}, the map $\tilde f_*$ is an isomorphism on $\pi_4 \left( \Lp
        \tau_{\leq 2p+1} \BSL_p(\CC)\right)$, and therefore by Corollary~\ref{l:lemma2},
        $\pi_{2p+1}(\Lp\tilde f)$ is an isomorphism.

        Since the projection $\Lp \tau_{\leq 2p+1} \BSL_p(\CC) \to \Lp \tau_{\leq 2p+1}
        \BPGL_p(\CC)$ induces an isomorphism on all higher homotopy groups $\pi_i$
        where $i\ge 3$, it follows that $f_*$ is an isomorphism on $\pi_{2p+1}\left(\Lp
        \tau_{\leq 2p+1} \BPGL_p(\CC)\right)$, as claimed.  
    \end{proof}
\end{theorem}

\section{Purity}

We consider purity in this section, giving two applications of algebraic topology to
algebraic purity questions. The first uses the machinery of Section~\ref{sec:topology} to show
that purity fails in general for $\PGL_p$ torsors, while the second uses~\cite{aw3}*{Theorem
D} to show that purity fails for the cohomological filtration on the Witt group.

\subsection{Definitions}

Let $\Fscr:\mathscr{C}^{\op}\rightarrow\mathrm{Sets}$ be a presheaf on some category of schemes
$\mathscr{C}$. We will suppress any mention of the category $\mathscr{C}$ throughout, and we will assume that all
necessary localizations of an object $X$ in $\mathscr{C}$ are also in $\mathscr{C}$. Suppose that $X$ is a regular
noetherian integral scheme in $\mathscr{C}$, and let $K$ be the function field of $X$. If the natural map
\begin{equation*}
    \im(\Fscr(X)\rightarrow\Fscr(\Spec K))\rightarrow \bigcap_{P\in
        X^{(1)}}\im(\Fscr(\Spec\Oscr_{X,P})\rightarrow\Fscr(\Spec K))
\end{equation*}
is a bijection, where $X^{(1)}$ denotes the set of codimension $1$ points of $X$, then we
say that \emph{purity} holds for $\Fscr(X)$.

\begin{example}\label{ex:br}
    If $X$ is a regular noetherian integral scheme with an ample line bundle such that
    $\QQ\subseteq\Gamma(X,\Oscr_X)$, then purity holds for $\Br(X)$. In
    particular, purity holds for $\Br(X)$ for smooth quasi-projective schemes over
    field of characteristic $0$.
    This follows from two facts. First, it is a theorem of Gabber and de
    Jong~\cite{dejong-gabber}
    that if $X$ has an ample line bundle, then $\Br(X)=\Hoh^2_{\et}(X,\Gm)_{\tors}$.
    Second, Gabber has shown (see Fujiwara~\cite{fujiwara}) that $\Hoh^2_{\et}(X,\Gm)_{\tors}$ satisfies purity
    when $X$ is a regular scheme and when each positive integer is invertible in $X$.
    The case of smooth affine schemes over fields had been handled previously by
    Hoobler~\cite{hoobler}, following Auslander and Goldman's
    work on the $2$-dimensional affine situation~\cite{auslander-goldman}*{Proposition 6.1},
    while Gabber~\cite{gabber-note} had proved the result in characteristic $0$
    with an added excellence condition.
    Gabber~\cite{gabber} proved purity for $\Hoh^2_{\et}(X,\Gm)_{\tors}$ without the excellence
    hypothesis when $\dim X\leq 3$; hence, in combination with the $\Br=\Br'$ result above,
    purity holds for the Brauer group when $\dim X\leq 3$ and $X$ has an ample line bundle.
    If $X$ is an arbitrary regular noetherian integral scheme, then
    purity holds for $\Br(X)'$, the part of the Brauer group containing the $m$-torsion for all
    $m>0$ invertible in $X$. This
    follows from purity for $\Hoh^2_{\et}(X,\mu_n)$ when $n$ is prime to $p$.
    See Fujiwara~\cite{fujiwara} together with \cite{sga4-3}*{Expos\'ee XIV, Section 3}
    or~\cite{colliot-thelene-birational}*{Theorem 3.8.2}.
\end{example}

Currently unknown is whether purity holds for $\Br(X)$ for every regular noetherian integral scheme $X$.
The results above should be contrasted to what happens for degree $3$ cohomology classes:
for any integer $n>1$,
there are smooth projective complex varieties $X$ such that purity fails for
$\Hoh^3_{\et}(X,\ZZ/n)$. See~\cite{colliot-thelene-voisin}*{Section 5}
for an overview, or Totaro~\cite{totaro-torsion} and Schoen~\cite{schoen} for examples. It is not hard
to see that unramified cohomology is homotopy invariant~\cite{totaro-injectivity}*{Theorem
1.3}, so it follows by using Jouanolou's device~\cite{jouanolou} that there are smooth
affine complex varieties where purity fails for $\Hoh^3_{\et}(X,\ZZ/n)$ as well.

\subsection{Purity for torsors}\label{sec:pt}

Let $X$ be a regular noetherian integral scheme, and let $G$ be a smooth reductive group scheme over $X$.
In~\cite{colliot-thelene-sansuc}*{Question 6.4}, Colliot-Th\'el\`ene and Sansuc ask whether purity holds for $\Hoh^1_{\et}(X,G)$. As stated
in the introduction, many examples are known where purity holds in the special case where $X=\Spec R$ is the spectrum of a regular
noetherian local ring $R$. But, as far as the authors are aware, except for our negative
results~\cite{aw4} for $G=\PGL_2$, no results are known in
the non-local case, either for or against purity, except in some trivial cases such as for special groups like $\SL_n$ and in the following
two theorems.

\begin{theorem}[\cite{colliot-thelene-sansuc}*{Corollaire 6.9}]
    Purity holds for $\Hoh^1_{\et}(X,G)$ for all regular noetherian integral schemes $X$ and
    all finite type $X$-group schemes of multiplicative type $G$.
\end{theorem}

\begin{theorem}[\cite{colliot-thelene-sansuc}*{Th\'eor\`eme 6.13}]\label{thm:cts2}
    Purity holds for $\Hoh^1_{\et}(X,G)$ for all regular noetherian integral $2$-dimensional
    schemes $X$ and all smooth reductive $X$-group schemes $G$.
\end{theorem}

Before we prove our main theorem, we need a standard result.

\begin{lemma}\label{lem:dvr}
    Let $R$ be a discrete valuation ring, and let $\alpha\in\Br(R)\subseteq\Br(K)$ be a Brauer class. If
    $D$ is a central simple algebra over $K$, the fraction field of $R$, with Brauer class
    $\alpha$, then every maximal order $A$ in $D$ is Azumaya over $R$.
    \begin{proof}
        A maximal order $A$ is in particular reflexive. Since a reflexive module on a
        regular domain of dimension at most $2$ is projective, $A$ is projective. The
        lemma now follows from the argument in the second paragraph of the proof
        of~\cite{auslander-goldman}*{Proposition 7.4}.
    \end{proof}
\end{lemma}

The goal of this paper is to show that Theorem~\ref{thm:cts2} does not extend to higher-dimensional schemes.
The method is based on~\cite{aw4}, augmented by the results of
Section~\ref{sec:topology}.

Let $a, b$ be positive integers and let $\P(a,ab)$ denote the complex algebraic group $\SL_{ab}(\CC)/\mu_a$. There is a
commutative diagram of short exact sequences of groups
\[ 
\xymatrix{ 1 \ar[r] & \mu_a  \ar[d] \ar[r] & \SL_{ab}(\CC) \ar[d] \ar[r] & \P(a,ab) \ar[r] \ar[d] & 1 \\
1\ar[r] & \CC^* \ar[r] & \GL_{ab}(\CC) \ar[r] & \PGL_{ab}(\CC) \ar[r] & 1. } \]
and therefore, for any topological space $X$, a commutative square
\begin{equation} \label{eq:ators}\xymatrix{\Hoh^1(X, \P(a,ab)) \ar[r] \ar[d]  & \Hoh^2(X, \ZZ/a) \ar[d] \\ \Hoh^1(X, \PGL_{ab}(\CC)) \ar[r] & \Hoh^2( X, \CC^*)
\iso \Hoh^3(X, \ZZ) }.  \end{equation}
If we have a principal $\P(a,ab)$-bundle on a topological space $X$, then the quotient map $\P(a,ab) \to \PGL_{ab}(\CC)$
gives rise to a principal $\PGL_{ab}(\CC)$-bundle and therefore a degree $ab$ topological Azumaya algebra. Diagram \ref{eq:ators}
implies that this Azumaya algebra is of exponent dividing $a$.

Similarly, in the category of schemes over $\CC$, an $\SL_{ab}/\mu_a$-torsor (for the \'etale topology) gives rise to a
degree $ab$ Azumaya algebra, and the exponent of this Azumaya algebra divides $a$.

We rely on the following argument repeatedly: If $X$ is a simply connected topological space, then $\pi_2(X) \iso
\Hoh_2(X, \ZZ)$ by the Hurewicz theorem. Then, by the universal coefficient theorem, the torsion $\Br(X) = \Hoh^3(X, \ZZ)_{\tors}$
is naturally the dual of the torsion subgroup of $\Hoh_2(X, \ZZ) \iso \pi_2(X)$. In the cases we consider, $\pi_2(X)$ is itself a
torsion abelian group, and therefore $\Br(X)$ is naturally the dual of $\pi_2(X)$.

\begin{lemma}\label{lem:paab}
  Suppose 
$f :X \to \BP(a,ab)$ is a $3$-equivalence of topological spaces. Denote the Azumaya algebra associated to the
$\PGL_{ab}(\CC)$ bundle classified by the composite $X \overset{f}{\to}  \BP(a,ab) \to \BPGL_{ab}(\CC)$ by $\Ascr(\CC)$.  
  The exponent of $\Ascr(\CC)$ is $a$.
\end{lemma}
\begin{proof}
  Since $f^*$ is a $3$-equivalence, the Hurewicz and universal coefficients theorems imply that it
  induces an isomorphism on $\Br_\topo( \cdot) \iso \Hoh^3(\cdot, \ZZ)_{\tors}$. It suffices therefore to show that the map
  $\phi^* : \Br_\topo(\BPGL_{ab}(\CC)) \to \Br_\topo( \BP(a,ab))$ takes a generator to a class of order $a$.

  The following is a diagram of short exact sequences of groups
  \[
  \xymatrix{
    1 \ar[r] & \mu_a \ar^i[d] \ar[r] & \SL_{ab}(\CC) \ar@{=}[d] \ar[r] & \P(a,ab) \ar^\phi[d] \ar[r] & 1 \\
    1 \ar[r] & \mu_a \ar[r] & \SL_{ab}(\CC) \ar[r] & \PGL_{ab}(\CC)\ar[r] & 1. }
  \]
  Here $i$ denotes the inclusion $\mu_a \subset \mu_{ab}$. This gives rise to a map of fiber sequences:
  \[
  \xymatrix{
   \BSL_{ab}(\CC) \ar@{=}[d] \ar[r] & \BP(a,ab) \ar^\phi[d] \ar[r] & \B^2 \mu_a \ar^{\B^2 i}[d] \\
   \BSL_{ab}(\CC) \ar[r] & \BPGL_{ab}(\CC) \ar[r] & \B^2 \mu_{ab}.
   }
  \]
  Since $\widetilde \Hoh^*(\SL_{ab}(\CC), \ZZ)$ vanishes below degree $4$, the natural map of 
  Serre spectral sequences for $\Hoh^*(\cdot, \ZZ)$ yields to a commutative square
  \[ \xymatrix{
  \Br_\topo( \BP(a,ab)) \ar@{=}[r] & \Hoh^3( \BP(a,ab), \ZZ) & \ar^\iso[l] \Hoh^3( \B^2 \mu_a, \ZZ) \iso \ZZ/a \\
  \Br_\topo(\BPGL_{ab}(\CC)) \ar@{=}[r] & \Hoh^3(\BPGL_{ab}(\CC), \ZZ) \ar^{(\B \phi)^*}[u] & \ar^\iso[l] \Hoh^3( \B^2 \mu_ab, \ZZ) \iso \ZZ/(ab) \ar^{(\B^2 i)^*}[u].}
  \]
  Where $(\B^2 i)^*$, by means of the Hurewicz and universal coefficient theorems, is seen to be the dual of the
  inclusion $\ZZ/a  = \mu_a \subset \mu_{ab} = \ZZ/(ab)$. Namely, it is a surjection $\ZZ/(ab) \to \ZZ/a$, as required.
\end{proof}

We now can prove our main theorem:

\begin{theorem}\label{thm:purity}
    Let $p$ be a prime. There exists a smooth affine complex variety $X$ of dimension
    $2p+2$ such that purity fails for $\Hoh^1_{\et}(X,\PGL_p)$.
\end{theorem}
 \begin{proof}
      Let $q>1$ be an integer prime to $p$. Let $V$ be an algebraic representation of the
      complex algebraic group $G = \SL_{pq}/\mu_p$ such that there is a $G$-invariant closed
      subvariety $S$ of codimension at least $p+2$ with the following properties: the
      complement $V-S$ is contained in the stable locus of the $G$-action on $V$ (in the
      sense of~\cite{mumford-fogarty-kirwan}) for some $G$-linearization of $\Oscr_V$, and
      $G$ acts freely on $V-S$. Such a representation is
      constructed in~\cite{totaro}*{Remark 1.4} by taking a large direct sum of any faithful
      $G$-representation. There is a universal geometric quotient
      $q:(V-S)\rightarrow (V-S)/G$ with $(V-S)/G$ a quasi-projective variety~\cite{mumford-fogarty-kirwan}. Moreover,
      $(V-S) \to (V-S)/G$ is an algebraic principal $G$ bundle, and $(V-S)/G$ is smooth,
      since $(V-S)\to(V-S)/G$ is a smooth surjective morphism with $(V-S)$ smooth. We can replace
      $(V-S)/G$ by an affine scheme using Jouanolou's device~\cite{jouanolou}, and then we can use the affine
      Lefschetz theorem~\cite{goresky-macpherson}*{Introduction, Section 2.2} to cut down to a $2p+2$-dimensional closed
      subscheme $X$. Pulling $q$ back along $X \to (V-S)/G$ gives an algebraic $G$-torsor $E \to
      X$, and therefore an induced $\PGL_{pq}$-torsor $E \times_{\P(p,pq)} \PGL_{pq}$, and finally an associated
      (algebraic) Azumaya algebra $\Ascr$ on $X$. Write $\alpha \in \Br(X)$ for the class of $\Ascr$; since $\Ascr$ is
      induced from a principal $\SL_{pq}/\mu_p$-bundle, the exponent of $\alpha$ divides $p$.

      The map $X \to (V-S)/G$ is an affine vector bundle, and upon complex realization, yields a homotopy
      equivalence. The realization $G(\CC)$ is the group $\P(p, pq)$, and by construction $((V-S)/G)(\CC) \to
      \BP(p,pq)$ is a $2p+3$ equivalence. The topological Azumaya algebra classified by the composite $X(\CC) \to
      \BP(p,pq) \to \BPGL_{pq}(\CC)$ is $\Ascr(\CC)$, and by Lemma \ref{lem:paab} it has exponent $p$. Since there is a
      homomorphism $\Br(X) \to \Br(X(\CC))$ taking the class, $\alpha$, of $\Ascr$ to that of $\Ascr(\CC)$, it follows
      that the exponent of $\alpha$ is exactly $p$.

      Returning to algebra, let $K$ be the function field of $X$. There is an inclusion $\Br(X) \subset \Br(K)$, and the class
      $\alpha \in \Br(K)$ corresponds to a central simple algebra $\Ascr \tensor_{\mathcal{O}_X} K$ of degree $pq$ and
      exponent $p$. From the theory of the index of a Brauer class of a field, ~\cite{gille-szamuely}*{Proposition 4.5.13}, we know that there is an Azumaya algebra
      $A'$ of degree $p$ (in fact, a division algebra) over $K$ in the class of $\alpha$.  By Lemma~\ref{lem:dvr},
      therefore, every
      codimension $1$ local ring $\Oscr_{X,x}$ of $X$ has the property that there is some Azumaya algebra of degree $p$
      representing the class of $\alpha$ in $\Br(\Oscr_{X,x})$, which is to say that the class of $A'$ in $\Hoh_\et^1(K,
      \PGL_p)$ lies in the intersection
      \[ \bigcap_{x \in X^{(1)}} \im \left(\Hoh_\et(\Spec \Oscr_{X,x}, \PGL_p) \to \Hoh^1_\et(\Spec K, \PGL_p)\right). \]

      To show that purity does not hold for $\PGL_p$, therefore, it suffices to show that $\alpha \in \Br(X)$ is not
      represented by any Azumaya algebra of degree $p$. By comparison, it is sufficient to show that the class of
      $\Ascr(\CC)$ in $\Br(X(\CC))$, is not represented by any topological Azumaya algebra of degree $p$. 

      Suppose for the sake of contradiction that such a topological Azumaya algebra exists. Let $f: X(\CC) \to \BPGL_p(\CC)$
      be a map classifying it. Since the class of $\Ascr(\CC)$ in $\Br(X(\CC))$ is of exponent $p$, it follows that the
      map $f^* : \Br(\BPGL_p(\CC)) \to \Br(X(\CC))$ is nonzero, and by the universal coefficients and Hurewicz
      theorems, it follows that the map $f_* : \ZZ/p \iso \pi_2(X(\CC)) \to \pi_2(\BPGL_p(\CC)) \iso \ZZ/p$ is nonzero,
      and in particular is an isomorphism.

        We consider the composition
        \begin{equation*}
            \tau_{\leq 2p+2}\BPGL_p(\CC)\rightarrow\tau_{\leq
            2p+2}\BP(p,pq)\rightarrow\tau_{\leq 2p+2}X\rightarrow\tau_{\leq
            2p+2}\BPGL_p(\CC),
        \end{equation*}
        where the first arrow is the $2p+2$-truncation of the $q$-fold block sum map
        $\BPGL_p(\CC)\rightarrow\BP(p,pq)$,
        the second arrow is a homotopy inverse to the homotopy equivalence
        $\tau_{\leq 2p+2}X\rightarrow\tau_{\leq 2p+2}\BP(p,pq)$, and the third arrow is the truncation of $f$. This composition
        induces an isomorphism on $\pi_2$, and hence on $\pi_{2p+1}$, by Theorem~\ref{thm:p}
        (applied to the further truncation $\tau_{\leq 2p+1}$ of the composition).
        But $\pi_{2p+1}\BP(p,pq)=0$, which is a contradiction.      
    \end{proof}

The theorem implies in particular that on $X$ there is an unramified degree-$p$ division algebra over $K$
that does not extend to an Azumaya algebra on $X$. The case $p=2$ was proved first
in~\cite{aw4}.

\begin{corollary}
    Let $p$ be a prime There exists a smooth affine complex variety $X$ of dimension
    $2p+2$ and an unramified division algebra $D$ over $\CC(X)$ of degree $p$ that contains
    no Azumaya maximal order on $X$.
\end{corollary}

\begin{scholium}\label{thm:patterns}
    Let $p$ be a prime, and let $n_1,\ldots,n_k$ be integers greater than $p$ such that $\gcd_i\{n_i\}=p$.
    There is a smooth affine complex variety $X$ of dimension $2p+2$ and a Brauer class $\alpha\in\Br(X)$
    of exponent $p$ such that there are Azumaya algebras of degrees $n_1,\ldots,n_k$ in the
    class $\alpha$, but no Azumaya algebra of degree $p$.
    \begin{proof}
        The proof is largely the same as that of the theorem, but using the algebraic group
        \begin{equation*}
            \SL_{n_1}\times\cdots\times\SL_{n_k}/\mu_p,
        \end{equation*}
        where $\mu_p$ is embedded diagonally in each of the groups $\SL_{n_i}$.
    \end{proof}
\end{scholium}

\subsection{Local purity}\label{sec:ojanguren}

In contrast to the global failure of purity for $\PGL_p$-torsors exhibited above,
in this section, we give a proof that purity holds for $\Hoh^1_{\et}(X,\PGL_n)$ when
$X$ is the spectrum of a regular local ring $R$ and the Brauer class has exponent invertible in
$X$. Our result is a minor generalization of a recent theorem of Ojanguren~\cite{ojanguren}
and of the local purity result for $\PGL_n$ in characteristic $0$ due to
Panin~\cite{panin-purity}.

To prove the theorem, we recall first a
result of DeMeyer, which is also used by both Ojanguren and Panin.

\begin{theorem}[DeMeyer~\cite{demeyer}*{Corollary 1}]
    Suppose that $R$ is an integral semi-local ring and that $\alpha\in\Br(R)$. Then, there exists a
    unique Azumaya algebra $A$ with class $\alpha$ having no idempotents besides $0$ and
    $1$. Moreover, any other Azumaya algebra with class $\alpha$ is of the form $\mathrm{M}_n(A)$
    for some $n$.
\end{theorem}

Now, we prove our local purity result. Define $\Hoh^1(X,\PGL_n)'$ to be the set of
$\PGL_n$-torsors whose associated Brauer class in $\Br(X)$ has exponent invertible in $X$.

\begin{theorem}\label{thm:ojanguren}
    Suppose that $R$ is a regular noetherian integral semi-local ring.
    Then, purity holds for $\Hoh^1(\Spec R,\PGL_n)'$.
    \begin{proof}
        Let $K$ be the function field of $R$, and let $D$ be a degree $n$ central simple
        algebra in
        \begin{equation*}
            \bigcap_{\mathrm{ht}\,P=1}\im\left(\Hoh^1_{\et}(\Spec
            R_p,\PGL_n)'\rightarrow\Hoh^1_{\et}(\Spec K,\PGL_n)'\right).
        \end{equation*}
        Let $m$ be the exponent of $[D]\in\Br(K)$. Because $D$ lifts to every codimension $1$
        local ring, so does the Brauer class. Since $m$ is invertible in $R$ and hence in
        these local rings, this Brauer class lifts to a Brauer class $\alpha\in\Br(R)$, by
        purity for $\Br(R)'$ (see Example~\ref{ex:br}).

        By DeMeyer's theorem, there exists an Azumaya algebra $A$ with Brauer class
        $\alpha$ such that every other Azumaya algebra in the class $\alpha$ is isomorphic to
        $\Mrm_r(A)$ for some $r$. In particular, $\ind(\alpha)=\deg(A)$, where,
        if $X$ is a scheme and $\alpha\in\Br(X)$, we define $\ind(\alpha)$ to be the gcd of the
        degrees of all Azumaya algebras with class $\alpha$.
        On the other hand, by~\cite{aw3}*{Proposition 6.1}, the index of $\alpha$ can be computed
        either over $R$ or over $K$. Thus, $\ind(\alpha)$ divides $\deg(D)$. Therefore,
        $D\iso\mathrm{M}_r(A_K)$ for some integer $r>0$. It follows that $\mathrm{M}_r(A)$
        is a class in $\Hoh^1_{\et}(\Spec R,\PGL_n)'$ that restricts to $D$, which shows that
        purity holds for $\Hoh^1_{\et}(\Spec R,\PGL_n)'$.
    \end{proof}
\end{theorem}

\subsection{Canonical factorization}\label{sec:canonical}

We prove in this section a theorem we view as evidence for Conjecture~\ref{conj:failure} for
all $\PGL_n$. 

Let $m>1$ divide $n$.
Both $\BP(m,n)$ and $\BPGL_m(\CC)$ are equipped with canonical maps to $K(\ZZ/m,2)$. Moreover, a
topological $\PGL_n(\CC)$ bundle,  $P\rightarrow X$, may be lifted to a $\P(m,n)$ bundle if and only if the associated obstruction
class $\delta_n(P)$ in $\Hoh^2(X,\ZZ/n)$ is $m$-torsion. A canonical factorization of Azumaya
algebras with structure group $\P(m,n)$ is a factorization $\BP(m,n)\rightarrow\BPGL_m\rightarrow K(\ZZ/m,2)$. The existence of such a
factorization would give, for every Azumaya algebra $A$ of degree $n$ and $m$-torsion
obstruction class, a canonical Azumaya algebra $B$ of degree $m$ with the same
obstruction class in $\Hoh^2(X,\ZZ/m)$. Unsurprisingly, this cannot occur.

\begin{theorem} \label{thm:fantasytheorem}
    If $n>m$, then there is no canonical factorization $\BP(m,n)\rightarrow\BPGL_m(\CC) \to K(\ZZ/m,2)$.
    \begin{proof}
      Suppose that $\BP(m,n)\rightarrow K(\ZZ/m,2)$ factors through $\BPGL_m(\CC)\rightarrow
      K(\ZZ/m,2)$. Let $\BPGL_m(\CC)\rightarrow\BP(m,n)$ be the
      map induced block-summation. Write $f:\BP(m,n)\rightarrow\BP(m,n)$
      for the composition. This map induces an isomorphism
      $\Hoh^2(\BP(m,n),\ZZ/m)\iso\ZZ/m$, and is in particular not nullhomotopic.

      As $\BP(m,n)$ is homotopy equivalent to $\mathrm{BSU}_n/\mu_m$, 
      there is a complete description of the homotopy-classes of
      self-maps $\BP(m,n) \to \BP(m,n)$ due to Jackowski, McClure, and Oliver~\cite{jackowski-mcclure-oliver-1}*{Theorem 2}.
      Their theorem says we can factor $f$ as $\Brm\alpha\circ\psi^k$, where $\alpha$ is an
      outer automorphism of $P(m,n)$, and
      $\psi^k$ is an unstable Adams operation on $\BP(m,n)$, for some $k\geq 0$ prime to the
      order of the Weyl group of $P(m,n)$, which is $n!$.  The map $\psi^k$
      induces multiplication by $k^i$ on $\Hoh^{2i}(\BP(m,n),\QQ)$. In particular a map $\BP(m,n) \to \BP(m,n)$ is either nullhomotopic or
      induces an isomorphism on rational cohomology.

      The rational cohomology of $\BP(m,n)$ is
        \begin{equation*}
            \Hoh^*(\BP(m,n),\QQ)\iso\QQ[c_2,\ldots,c_n], \quad c_i \in \Hoh^{2i}( \BP(m,n), \QQ),
        \end{equation*}
        while that of $\BPGL_m(\CC)$ is
        \begin{equation*}
            \Hoh^*(\BPGL_m(\CC),\QQ)\iso\QQ[c_2,\ldots,c_m], \quad c_i \in \Hoh^{2i}( \BP(m,n), \QQ).
        \end{equation*}
        In particular,
        \begin{equation*}
            \dim\Hoh^{2m+2}(\BP(m,n),\QQ)=\dim\Hoh^{2m+2}(\BPGL_m(\CC),\QQ)+1,
        \end{equation*}
        so that $f$ cannot induce an isomorphism on rational cohomology, and must be nullhomotopic, a contradiction.
    \end{proof}
\end{theorem}

The argument above has philosophically informed the authors' work both in this paper and in \cite{aw4}.
In order to construct algebraic counterexamples, however, we must use complex algebraic
varieties $X$ that approximate $\BPGL_m(\CC)$ in the sense that
there exists a map $X(\CC) \to \BPGL_m(\CC)$ induced by an algebraic $\PGL_m$-torsor on $X$
and inducing an isomorphism on homotopy groups in a range
dimensions, and for these we cannot bring the strength of~\cite{jackowski-mcclure-oliver-1} to bear. We have made do with \textit{ad hoc}
arguments that furnish obstructions in known, bounded dimension to maps $\BPGL_m(\CC) \to \BPGL_m(\CC)$. For instance, the topological
plank in the argument proving that purity fails for $\Hoh^1_{\et}(X,\PGL_p)$ is an obstruction to a map
\begin{equation*}
    \tau_{\leq 2p+1}\BP(p,pq) \to \tau_{\leq 2p+1}\BPGL_p(\CC)
\end{equation*}
that induces an isomorphism on Brauer group. This obstruction depends on Theorem
\ref{thm:p}, which describes a restriction on maps
\[ \tau_{\le 2p+1} \BPGL_p \to \tau_{\le 2p+1} \BPGL_p, \]
in that it says a map inducing an isomorphism on $\pi_2$ must necessarily also induce an isomorphism on the $p$--primary part of
$\pi_{2p+1}$. To prove Conjecture \ref{conj:failure} for all $\PGL_m$, one might only have
to find an obstruction to the existence of
maps $X \to \BPGL_m$ where $X$ approximates $\BP(m,mq)$, with $q>1$ prime to $m$.

\subsection{The Witt group}\label{sec:witt}

Our second application of topology to purity is to give a new example where purity fails for
the cohomological filtration on the Witt group.

\begin{example}
    Local purity is known for the Witt group $\W(\Spec R)$ whenever $R$ is a regular
    noetherian local ring containing a field
    of characteristic not $2$ by work of Ojanguren and
    Panin~\cite{ojanguren-panin}.
\end{example}

Given the positive results for the Brauer group, it is natural to ask the following question.

\begin{question}
    Does purity hold for $\W(X)$ when $X$ is an regular excellent noetherian integral scheme
    having no points of characteristic $2$?
\end{question}

It is known that purity holds for $\W(X)$ when $X$ is a regular noetherian separated
integral scheme of Krull dimension at most $4$ and
$2$ is invertible in $\Gamma(X,\mathscr{O}_X)$ by
Balmer-Walter~\cite{balmer-walter}*{Corollary~10.3}. However,
Totaro~\cite{totaro-injectivity} showed that the injectivity property fails for the Witt
group: there is a smooth affine complex $5$-fold such that $\W(X)\rightarrow\W(K)$ is
not injective. Thus, it might be natural to guess that the purity property fails as well.
For an extensive overview of results on purity for the
Witt group, see Auel~\cite{auel-remarks}.

Let $\I^1(X)$ be the ideal of $\W(X)$ generated by even-dimensional quadratic spaces. There
is a discriminant map $\I^1(X)\rightarrow\Hoh^1_{\et}(X,\mu_2)$. Let $\I^2(X)$ be the
kernel. There is a map $\I^2(X)\rightarrow{_2\Br(X)}$, called the Clifford invariant map.
Denote by $\I^3(X)$ the kernel. It is known that purity fails for $\I^2(X)/\I^3(X)$. The first examples were
due to Parimala and Sridharan~\cite{parimala-sridharan-2}, who showed that it fails for some
affine bundle torsors over smooth projective $p$-adic
curves. We include another example below, which uses a smooth affine
variety we constructed in~\cite{aw3}, giving the first examples over $\CC$.

\begin{example} \label{ex:wittpurity}
    Let $X$ be the smooth affine $5$-dimensional variety over $\CC$ constructed
    in~\cite{aw3}*{Theorem D}, having a Brauer class $\alpha\in\Br(X)$ of exponent $2$ that is not in
    the image of the Clifford invariant map $\I^2(X)\rightarrow{_2\Br(X)}$. Consider the
    commutative diagram
    \begin{equation*}
        \xymatrix{
            & 0\ar[d] & 0\ar[d] & 0\ar[d]\\
            & \I^3(X)\ar[r]\ar[d] & \I^3(K)\ar[r]\ar[d] & \bigoplus_{p\in X^{(1)}}\I^2(k(p))\ar[d]\\
            & \I^2(X)\ar[r]\ar[d] & \I^2(K)\ar[r]\ar[d] & \bigoplus_{p\in X^{(1)}}\I^1(k(p))\ar[d]\\
            0\ar[r] & _2\Br(X)\ar[r] & _2\Br(K)\ar[r]\ar[d] & \bigoplus_{p\in X^{(1)}}\Hoh^1(k(p),\ZZ/2)\ar[d]\\
            & & 0 & 0\\
        }
    \end{equation*}
    where the columns and the bottom row are exact, and where $\I^2(X)$ (resp. $\I^3(X)$)
    maps into the kernel of the map $\I^2(K)\rightarrow\bigoplus\I^1(k(p))$ (resp
    $\I^3(K)\rightarrow\bigoplus\I^2(k(p))$). The image of $\alpha$ in $_2\Br(K)$
    is in the image of the map $\I^2(K)\rightarrow{_2\Br(K)}$ by Merkurjev's theorem; say it is the Clifford
    invariant of $\sigma\in\I^2(K)$. Then, $\sigma$ is unique up to an element of $\I^3(K)$. On
    the other hand, the ramification classes $\partial_p(\sigma)$ are all in $\I^2(k(p))$.
    Hence, $\overline{\sigma}\in\I^2(K)/\I^3(K)$ is unramified. But, by construction, it is not
    in the image of $\I^2(X)/\I^3(K)\rightarrow\I^2(K)/\I^3(K)$.
\end{example}

This is the first such example known for a variety over an algebraically closed field. It
has the added advantage that it is not explained by the presence of line-bundle
valued quadratic forms, as explained in~\cite{aw3}*{Section 7}.


\begin{bibdiv}
\begin{biblist}


\bib{aw1}{article}{
    author={Antieau, Benjamin},
    author={Williams, Ben},
    title={The period-index problem for twisted topological $K$-theory},
    journal={Geom. Topol.},
    volume={18},
    date={2014},
    number={2},
    pages={1115--1148},
    issn={1465-3060},
}

\bib{aw3}{article}{
    author={Antieau, Benjamin},
    author={Williams, Ben},
    title={The topological period-index problem over 6-complexes},
    journal={J. Topol.},
    volume={7},
    date={2014},
    number={3},
    pages={617--640},
    issn={1753-8416},
}

\bib{aw4}{article}{
    author={Antieau, Benjamin},
    author={Williams, Ben},
    title={Unramified division algebras do not always contain Azumaya maximal orders},
    journal={Invent. Math.},
    volume={197},
    date={2014},
    number={1},
    pages={47--56},
    issn={0020-9910},
}

\bib{sga4-3}{book}{
    title={Th\'eorie des topos et cohomologie \'etale des sch\'emas. Tome 3},
    series={Lecture Notes in Mathematics, Vol. 305},
    note={S\'eminaire de G\'eom\'etrie Alg\'ebrique du Bois-Marie 1963--1964 (SGA 4);
    Dirig\'e par M. Artin, A. Grothendieck et J. L. Verdier. Avec la collaboration de P. Deligne et B. Saint-Donat},
    publisher={Springer-Verlag},
    place={Berlin},
    date={1973},
    pages={vi+640},
}



\bib{auel}{article}{
    author = {Auel, Asher},
    title = {Surjectivity of the total Clifford invariant and Brauer dimension},
    journal = {ArXiv e-prints},
    eprint = {http://arxiv.org/abs/1108.5728},
    year = {2011},
}

\bib{auel-remarks}{article}{
    author={Auel, Asher},
    title={Remarks on the Milnor conjecture over schemes},
    conference={
        title={Galois-Teichmüller Theory and Arithmetic Geometry (Kyoto, 2010)},
    },
    book={
        series={Advanced Studies in Pure Mathematics},
        volume={63},
        editor={Nakamura, H.},
        editor={Pop, F.},
        editor={Schneps, L.},
        editor={Tamagawa, A.},
    },
    year={2012},
    pages={1--30},
}


\bib{auslander-goldman}{article}{
    author={Auslander, Maurice},
    author={Goldman, Oscar},
    title={The Brauer group of a commutative ring},
    journal={Trans. Amer. Math. Soc.},
    volume={97},
    date={1960},
    pages={367--409},
    issn={0002-9947},
}

\bib{balmer-walter}{article}{
    author={Balmer, Paul},
    author={Walter, Charles},
    title={A Gersten-Witt spectral sequence for regular schemes},
    journal={Ann. Sci. \'Ecole Norm. Sup. (4)},
    volume={35},
    date={2002},
    number={1},
    pages={127--152},
    issn={0012-9593},
}



\bib{bott}{article}{
    author={Bott, Raoul},
    title={The space of loops on a Lie group},
    journal={Michigan Math. J.},
    volume={5},
    date={1958},
    pages={35--61},
    issn={0026-2285},
}

\bib{bousfield-kan}{book}{
    author={Bousfield, A. K.},
    author={Kan, D. M.},
    title={Homotopy limits, completions and localizations},
    series={Lecture Notes in Mathematics, Vol. 304},
    publisher={Springer-Verlag},
    place={Berlin},
    date={1972},
    pages={v+348},
}


\bib{cartan}{article}{
    author = {Cartan, H.},
    title = {D{\'e}termination des alg{\`e}bres {$\Hoh_*(\pi,n;\ZZ)$}},
    journal = {S{\'e}minaire H. Cartan},
    volume = {7},
    number = {1},
    pages = {11-01--11-24},
    publisher = {Secr{\'e}tariat math{\'e}matique},
    address = {Paris},
    year = {1954/1955},
}

\bib{chernousov-panin}{article}{
    author={Chernousov, Vladimir},
    author={Panin, Ivan},
    title={Purity of $G_2$-torsors},
    journal={C. R. Math. Acad. Sci. Paris},
    volume={345},
    date={2007},
    number={6},
    pages={307--312},
    issn={1631-073X},
}

\bib{colliot-thelene-sansuc}{article}{
    author={Colliot-Th{\'e}l{\`e}ne, J.-L.},
    author={Sansuc, J.-J.},
    title={Fibr\'es quadratiques et composantes connexes r\'eelles},
    journal={Math. Ann.},
    volume={244},
    date={1979},
    number={2},
    pages={105--134},
    issn={0025-5831},
}

\bib{colliot-thelene-birational}{article}{
    author={Colliot-Th{\'e}l{\`e}ne, J.-L.},
    title={Birational invariants, purity and the Gersten conjecture},
    conference={
        address={Santa Barbara, CA},
        date={1992},
    },
    book={
        series={Proc. Sympos. Pure Math.},
        volume={58},
        publisher={Amer. Math. Soc.},
        place={Providence, RI},
    },
    date={1995},
    pages={1--64},
}


\bib{colliot-thelene-voisin}{article}{
    author={Colliot-Th{\'e}l{\`e}ne, Jean-Louis},
    author={Voisin, Claire},
    title={Cohomologie non ramifi\'ee et conjecture de Hodge enti\`ere},
    journal={Duke Math. J.},
    volume={161},
    date={2012},
    number={5},
    pages={735--801},
    issn={0012-7094},
}


\bib{dejong-gabber}{article}{
    author={de Jong, A. J.},
    title={A result of Gabber},
    eprint={http://www.math.columbia.edu/~dejong/},
}



\bib{demeyer}{article}{
    author={DeMeyer, F. R.},
    title={Projective modules over central separable algebras},
    journal={Canad. J. Math.},
    volume={21},
    date={1969},
    pages={39--43},
    issn={0008-414X},
}

\bib{ekedahl}{article}{
    author = {Ekedahl, Torsten},
    title = {Approximating classifying spaces by smooth projective varieties},
    journal = {ArXiv e-prints},
    eprint =  {http://arxiv.org/abs/0905.1538},
    year = {2009},
}

\bib{fedorov-panin}{article}{
    author = {Fedorov, Roman},
    author = {Panin, Ivan},
    title = {Proof of Grothendieck-Serre conjecture on principal bundles over regular local rings containing infinite fields},
    journal = {ArXiv e-prints},
    eprint =  {http://arxiv.org/abs/1211.2678},
    year = {2012},
}

\bib{fujiwara}{article}{
    author={Fujiwara, Kazuhiro},
    title={A proof of the absolute purity conjecture (after Gabber)},
    conference={
        title={Algebraic geometry 2000, Azumino (Hotaka)},
    },
    book={
        series={Adv. Stud. Pure Math.},
        volume={36},
        publisher={Math. Soc. Japan},
        place={Tokyo},
    },
    date={2002},
    pages={153--183},
}

\bib{gabber}{article}{
    author={Gabber, Ofer},
    title={Some theorems on Azumaya algebras},
    conference={
    title={The Brauer group},
    address={Sem., Les Plans-sur-Bex},
    date={1980},
    },
    book={
    series={Lecture Notes in Math.},
    volume={844},
    publisher={Springer},
    place={Berlin},
    },
    date={1981},
    pages={129--209},
}

\bib{gabber-note}{article}{
    author={Gabber, Ofer},
    title={A note on the unramified Brauer group and purity},
    journal={Manuscripta Math.},
    volume={95},
    date={1998},
    number={1},
    pages={107--115},
    issn={0025-2611},
}

\bib{gille-szamuely}{book}{
    author={Gille, Philippe},
    author={Szamuely, Tam{\'a}s},
    title={Central simple algebras and Galois cohomology},
    series={Cambridge Studies in Advanced Mathematics},
    volume={101},
    publisher={Cambridge University Press},
    place={Cambridge},
    date={2006},
    pages={xii+343},
    isbn={978-0-521-86103-8},
    isbn={0-521-86103-9},
}

\bib{goresky-macpherson}{book}{
    author={Goresky, Mark},
    author={MacPherson, Robert},
    title={Stratified Morse theory},
    series={Ergebnisse der Mathematik und ihrer Grenzgebiete (3)},
    volume={14},
    publisher={Springer-Verlag},
    place={Berlin},
    date={1988},
    pages={xiv+272},
    isbn={3-540-17300-5},
}

\bib{grothendieck-brauer-1}{article}{
    author={Grothendieck, Alexander},
    title={Le groupe de Brauer. I. Alg\`ebres d'Azumaya et interpr\'etations diverses},
    conference={
    title={S\'eminaire Bourbaki, Vol.\ 9},
    },
    book={
    publisher={Soc. Math. France},
    place={Paris},
    },
    date={1995},
    pages={Exp.\ No.\ 290, 199--219},
}

%
 
\bib{hoobler}{article}{
    author={Hoobler, Raymond T.},
    title={A cohomological interpretation of Brauer groups of rings},
    journal={Pacific J. Math.},
    volume={86},
    date={1980},
    number={1},
    pages={89--92},
    issn={0030-8730},
}


\bib{jackowski-mcclure-oliver-1}{article}{
    author={Jackowski, Stefan},
    author={McClure, James},
    author={Oliver, Bob},
    title={Homotopy classification of self-maps of $BG$ via $G$-actions. I},
    journal={Ann. of Math. (2)},
    volume={135},
    date={1992},
    number={1},
    pages={183--226},
    issn={0003-486X},
}


\bib{jouanolou}{article}{
    author={Jouanolou, J. P.},
    title={Une suite exacte de Mayer-Vietoris en $K$-th\'eorie alg\'ebrique},
    conference={
    title={Algebraic $K$-theory, I: Higher $K$-theories (Proc. Conf.,
    Battelle Memorial Inst., Seattle, Wash., 1972)},
    },
    book={
    publisher={Springer},
    place={Berlin},
    },
    date={1973},
    pages={293--316. Lecture
    Notes in Math., Vol.
    341},
}






\bib{mumford-fogarty-kirwan}{book}{
    author={Mumford, D.},
    author={Fogarty, J.},
    author={Kirwan, F.},
    title={Geometric invariant theory},
    series={Ergebnisse der Mathematik und ihrer Grenzgebiete (2) [Results in
    Mathematics and Related Areas (2)]},
    volume={34},
    edition={3},
    publisher={Springer-Verlag, Berlin},
    date={1994},
    pages={xiv+292},
    isbn={3-540-56963-4},
}

\bib{ojanguren}{article}{
    author = {Ojanguren, Manuel},
    title = {Wedderburn's theorem for regular local rings},
    journal = {Linear Algebraic Groups Preprint Server},
    eprint = {http://www.mathematik.uni-bielefeld.de/LAG/man/500.html},
    year = {2013},
}

\bib{ojanguren-panin}{article}{
    author={Ojanguren, Manuel},
    author={Panin, Ivan},
    title={A purity theorem for the Witt group},
    journal={Ann. Sci. \'Ecole Norm. Sup. (4)},
    volume={32},
    date={1999},
    number={1},
    pages={71--86},
    issn={0012-9593},
}

\bib{panin-purity}{article}{
    author={Panin, I. A.},
    title={Purity conjecture for reductive groups},
    journal={Vestnik St. Petersburg Univ. Math.},
    volume={43},
    date={2010},
    number={1},
    pages={44--48},
    issn={1063-4541},
}

\bib{parimala-sridharan-2}{article}{
    author={Parimala, R.},
    author={Sridharan, R.},
    title={Nonsurjectivity of the Clifford invariant map},
    note={K. G. Ramanathan memorial issue},
    journal={Proc. Indian Acad. Sci. Math. Sci.},
    volume={104},
    date={1994},
    number={1},
    pages={49--56},
    issn={0253-4142},
}




\bib{schoen}{article}{
    author={Schoen, Chad},
    title={Complex varieties for which the Chow group mod $n$ is not finite},
    journal={J. Algebraic Geom.},
    volume={11},
    date={2002},
    number={1},
    pages={41--100},
    issn={1056-3911},
}


\bib{totaro}{article}{
    author={Totaro, Burt},
    title={The Chow ring of a classifying space},
    conference={
    title={Algebraic $K$-theory},
    address={Seattle, WA},
    date={1997},
    },
    book={
    series={Proc. Sympos. Pure Math.},
    volume={67},
    publisher={Amer. Math. Soc.},
    place={Providence, RI},
    },
    date={1999},
    pages={249--281},
}

\bib{totaro-torsion}{article}{
    author={Totaro, Burt},
    title={Torsion algebraic cycles and complex cobordism},
    journal={J. Amer. Math. Soc.},
    volume={10},
    date={1997},
    number={2},
    pages={467--493},
    issn={0894-0347},
}

\bib{totaro-injectivity}{article}{
    author={Totaro, Burt},
    title={Non-injectivity of the map from the Witt group of a variety to the
    Witt group of its function field},
    journal={J. Inst. Math. Jussieu},
    volume={2},
    date={2003},
    number={3},
    pages={483--493},
    issn={1474-7480},
}

\bib{vezzosi}{article}{
   author={Vezzosi, Gabriele},
   title={On the Chow ring of the classifying stack of ${\rm PGL}_{3,\mathbf{C}}$},
   journal={J. Reine Angew. Math.},
   volume={523},
   date={2000},
   pages={1--54},
   issn={0075-4102},
}

\bib{vistoli}{article}{
    author={Vistoli, Angelo},
    title={On the cohomology and the Chow ring of the classifying space of ${\rm PGL}_p$},
    journal={J. Reine Angew. Math.},
    volume={610},
    date={2007},
    pages={181--227},
    issn={0075-4102},
}




\end{biblist}
\end{bibdiv}
\end{document}